\documentclass[11pt,reqno]{paper}
\usepackage[pdftex]{graphicx}
\usepackage{fancyhdr}
\usepackage[a4paper, hmargin=2.16cm, vmargin={3.2cm, 3.2cm}]{geometry}
\usepackage[utf8]{inputenc}

\usepackage{titlesec}
\titleformat{\subsection}[runin] 
{\normalfont\bfseries}{\thesubsection}{1em}{}
\titleformat{\subsubsection}[runin] 
{\normalfont\bfseries}{\thesubsubsection}{1em}{}
\titleformat{\paragraph}[runin] 
{\normalfont\bfseries}{\theparagraph}{1em}{}
\makeatletter
\def\@seccntformat#1{\@ifundefined{#1@cntformat}%
	{\csname the#1\endcsname\quad}  
	{\csname #1@cntformat\endcsname}
}
\let\oldappendix\appendix 
\renewcommand\appendix{%
	\oldappendix
	\newcommand{\section@cntformat}{\appendixname~\thesection\quad}
}
\makeatother
\usepackage[titletoc,title]{appendix}

\usepackage{enumerate}
\usepackage{amsmath}%
\usepackage{amsfonts}%
\usepackage{amsthm}
\usepackage{amssymb}%
\usepackage{graphicx}
\usepackage{hyperref}
\usepackage{theoremref}
\usepackage{bbm}
\usepackage{amssymb}
\usepackage{color}

\def\64{}
\def\256{4}
\def\512{8}

\DeclareMathOperator{\re}{Re}

\theoremstyle{plain}
\newtheorem{theorem}{Theorem}
\newtheorem{lemma}{Lemma}[section]
\newtheorem{prop}[lemma]{Proposition}

\theoremstyle{remark}

\newtheorem{Ex}{Example}
\theoremstyle{definition}
\newtheorem{definition}{Definition}

\begin{document}
\title{Power substitution in quasianalytic Carleman classes}
\author{Lev Buhovsky\textsuperscript{1}, Avner Kiro\textsuperscript{2} and Sasha Sodin\textsuperscript{3}}
\maketitle
\footnotetext[1]{School of Mathematical Sciences, Tel Aviv
	University, Tel Aviv, 69978, Israel. Email: levbuh@tauex.tau.ac.il. Supported in part by ERC Starting Grant 757585 and ISF Grant 2026/17.}
\footnotetext[2]{School of Mathematical Sciences, Tel Aviv
University, Tel Aviv, 69978, Israel. Email: avnerkiro@gmail.com. Supported in part by ERC Advanced Grant 692616, ISF Grant 382/15 and BSF Grant 2012037.}
\footnotetext[3]{School of Mathematical Sciences, Queen Mary University of London, London E1 4NS, United Kingdom \& School of Mathematical Sciences, Tel Aviv
University, Tel Aviv, 69978, Israel. Email:
a.sodin@qmul.ac.uk. Supported in part by the European
Research Council starting grant 639305 (SPECTRUM) and by a Royal Society Wolfson Research Merit Award.}

	\begin{abstract}
		Consider an equation of the form $f(x)=g(x^k)$, where $k>1$ is an integer and $f(x)$ is a function in a given  Carleman class of smooth functions. For each $k$, we construct a non-homogeneous Carleman-type class which contains all the smooth solutions $g(x)$ to such equations. We prove that if the original Carleman class is quasianalytic, then so is the new class. The results admit an extension to multivariate functions.
	\end{abstract}
	
\section{Introduction}
In this text, we consider power substitutions in Carleman classes, i.e.\ equations of the form $g(x^k)=f(x)$, where $k>1$ is an integer and  $f$ is a given function in a  quasianalytic Carleman class $C^M$   (see Definition~\ref{def:Cclass}).  Our motivation to study power substitutions in Carleman classes  mainly comes from \cite{Bierstone}. There, it was shown, under certain conditions, that if $F(x,y)$ belongs to a quasianalytic Carleman class $ C^M(\mathbb{R}^{d_1}\times \mathbb{R}^{d_2})$ (see Definition \ref{def: multCarlemanClass}) and the equation  $F(x,y)=0$ admits a $C^\infty $ solution $y=h(x)$, then $h$ is the image of a $C^M(\mathbb{R}^{d_1})$ function under finitely many power substitutions and blow-ups. Another source of motivation comes from \cite{BierstoneMilman,RolinSpeisseggerWilkie}, where normalization algorithms for power series in Carleman classes also require finitely many power substitutions and blow-ups.

The results of \cite{Bierstone} imply that smooth solutions $g$ of $f(x)=g(x^k)$ inherit a certain quasianalytic property from the original Carleman class: they are definable in an appropriate $o$-minimal structure. The combination of our main results (Theorems \ref{thm: thm1} and \ref{thm:thm2} below) implies the following more explicit quasianalytic property: $g$  belongs to a quasianalytic class $C_{1-1/k}^M$ (Definition~\ref{def:3}) completely characterized in terms of bounds on the derivatives of $g$.  
	
\begin{definition}\label{def:Cclass}
Let $M=(M_n)_{n\geq 0}$ be a positive sequence and let $I$ be an interval. The \textit{Carleman class} $C^M(I)$ consists of all functions $f\in C^\infty(I)$ such that, for any compact set $K\subset I$, there exist constants  $A,\;B>0$ such that 
\[\left|f^{(n)}(x)\right|\leq  A B^n M_n,\quad x\in K,\; n\geq 0. \]
A   Carleman class $C^M(I)$ is said to be \textit{quasianalytic}  if any $f\in C^M(I)$ that has a zero formal Taylor expansion  at some $x\in I$ is identically zero. 
\end{definition} 

According to the Denjoy--Carleman theorem   (see \cite{carleman} or \cite[\S12]{Mandelbrojt} for this exact formulation)  the class $C^M(I)$  is quasianalytic if and only if 
\begin{equation}\label{eq:dc} \sum_{n\geq 0} \frac{M_n^C}{M_{n+1}^C}=\infty~, \end{equation}
where $M^C$ is the largest log-convex minorant of $M$, i.e. 
$$M_n^C =\min\left\{M_n, \inf_{j<n<\ell} M_j^{(\ell-n)/(\ell-j)}M_\ell^{(n-j)/(\ell-j)}\right\}.$$
A necessary and sufficient condition for the equality of two Carleman classes was given in  \cite{cartan1940solution}. In particular, if the sequence $M$ satisfies $M_n\geq n!$ for any $n\geq0$, then $C^M(I)=C^{M^C}(I)$.  

Given $f\in C^M(I)$ where $I$ is an interval such that $0\in I$ (possibly as an endpoint), we consider a function $g$ defined on the interval $I^k=\{x^k:x\in I\}$ and satisfying $g(x^k)=f(x)$. It is well known that if the class $C^M(I)$ contains all real analytic functions (i.e. there exists $\delta>0$ such that  $M_n^C\geq\delta^{n+1} n!$ for every $n\geq0$) in $I$, then  $g\in C^M\left(I^k\setminus(-\varepsilon,\varepsilon)\right)$, for any  $\varepsilon>0$ (see  Lemma \ref{lem: add smooth lem} below),   but $g$ may be singular at zero. If $g$ happens to be   $ C^\infty$ near zero, then  
\begin{equation}\label{eq:deriv0}
g^{(n)}(0)/n!=f^{(kn)}/(kn)!~,
\end{equation} as follows (for polynomials) from the Cauchy theorem, and thus there exist constants $A,\;B>0$ such that 
\begin{equation}
\left|g^{(n)}(0)\right|\leq A B^n \frac{M_{kn}}{(kn)!^{1-1/k}},\quad n\geq 0,  \label{eq:triv g bound}
\end{equation}
It was shown in \cite{Bierstone} and \cite{Nowak}  that under some regularity conditions on the sequence $M$,   the  estimate \eqref{eq:triv g bound} on the derivatives of  $g$ at zero  can be extended to the interval $I^k$. A similar fact,  without regularity assumptions, follows from a combination of Theorem~\ref{thm: thm1} and  Proposition~\ref{prop: extra smooth}. Namely, it follows that $g\in C^{M^{(k)}}(I^k)$, where $M^{(k)}_n=n!\sup_{j\leq nk+1} \frac{M_{j}}{j!} $.  Note that by the formula (\ref{eq:deriv0}) for $g^{(n)}(0)$ there is no smaller Carleman class that contains $g$.  By the Denjoy--Carleman theorem, the classes $C^{M^{(k)}}$ may fail to be quasianalytic even if the original class $C^{M}$ is quasianalytic. We will show that in the above case, the function  $g$  belongs to a new non-homogeneous class $C^M_{1-1/k}(I^k)$ of smooth functions (defined in Definition~\ref{def:3} below) and that the latter class is quasianalytic.

\medskip

 Results similar to these were first obtained by the second author as a byproduct of the work \cite{Avner}. In the first version of this paper, available on arXiv under the same address, we applied the elementary method of Bang \cite{Bang2} to relax the regularity assumptions at the expense of relinquishing the precise asymptotics. Here, instead of adapting the arguments from classical quasianalyticity, we employ a reduction to the classical setting, and in this way relax the regularity assumption even further.

\section{Results}
\begin{definition}\label{def:3}
	Let $M$ be a a positive sequence,  $I$ be an interval and let $0\leq a$. The class $C^M_a(I)$ consists of all functions $g\in C^\infty(I)$ such that for any compact set $K\subset I$, there exist constants $A,\;B>0$ such that 
	\[ \left|g^{(n)}(x)\right|\leq A B^n \frac{M_n}{|x|^{an}},\quad x\in K\setminus\{0\},\; n\geq 0.\]
\end{definition}
\begin{theorem}\label{thm: thm1}
	Let $C^M(I)$ be  Carleman class, and let $k>1$ be an integer. Let  $g\in C^\infty(I^k)$, and let $f(x)=g(x^k)$. If        $f\in C^M(I)$, then $g\in C^M_a(I^k)$, where $a=1-\tfrac{1}{k}$.
\end{theorem}

The next  proposition demonstrates that functions in $C^M_a(I)$  with $a<1$   carry additional, implicit, control on their successive derivatives.
\begin{prop}{\label{prop: extra smooth}}
	Let $M$ be positive sequence, and let $k>1$ be an integer. If  $g\in C^M_a(I)$ with $a=1-\frac{1}{k}$, then $g\in C^{M^{(k)}}(I)  $, where 
	$$M^{(k)}_n=n!\sup_{j\leq nk+1} \frac{M_{j}}{j!}. $$
\end{prop}
In the case that $g$ is the smooth solution to a power substitution $g(x^k)=f(x)$ with $f\in C^M(I)$, this additional smoothness was already shown in \cite{Bierstone,Nowak}.  

Our next result is about quasianalyticity of $C^M_a(I)$ with $a<1$.
\begin{theorem}\label{thm:thm2} Let $M$ be positive sequence, and let $0\leq a<1$. If $M$ is log--convex or $(M_n/n!)_{n\geq0}$ is non decreasing, then the class $C^M_a(I)$ is quasianalytic if and only if \eqref{eq:dc} holds.
\end{theorem}

	There are multivariate analogues to Theorems 1 and 2. We postpone the discussion of such analogues to the last section. 

\medskip
Finally, the next two examples show that  when $a\geq 1$, there are no analogues statements  to Proposition \ref{prop: extra smooth} and Theorem \ref{thm:thm2}, even in the simple analytic case, when $M_n=n!$. 
\begin{Ex}\label{ex: Ex1}
	Consider the $C^\infty[0,1]$ function $g$, defined by  $g(x)=\exp(-1/x)$ for $0<x\leq 1$ (and $g(0)=0$). By Cauchy's  estimates for the derivatives of analytic functions, we have 
	$$|g^{(n)}(x)|\leq n!\frac{2^n}{x^n}\cdot\max_{|z-x|=\frac{x}{2}}|g(z)|\leq n!\frac{2^n}{x^n}. $$
	So $g\in C^{M}_1\left([0,1]\right)$ with $M_n=n!$, and $g^{(n)}(0)=0$ for any $n\geq 0$. In particular, there is no analogue to Theorem~\ref{thm:thm2} for $a\geq1$.
\end{Ex}
\begin{Ex}\label{ex: Ex2}
	Let $(N_n)_{n\geq0}$ be an arbitrary positive sequence. We argue that there exists  a function $g\in C^{n!}_1\left([0,1]\right)$ such that 
\begin{equation}
 {\limsup_{n\to\infty}}\frac{|g^{(n)}(0)|}{N_n}>0. \label{eq: func with large der}
\end{equation}
	In particular, the existence of such a function shows that the analogue to Proposition~\ref{prop: extra smooth} in the case $a\geq 1$ does not hold.
	
	The construction of $g$ is done in two steps. First, by Borel's Lemma (see  \cite[p. 44]{Borel} or \cite[p.16]{hormanderBook}) there is a $2\pi$ periodic and $C^\infty(\mathbb{R})$ function, $h$, such that  $h^{(n)}(0)=N_n$, for any $n\geq 0$.  Expanding the function $h$ in  a  Fourier series, we have 
	\[h(x)=\sum_{j\in \mathbb{Z}} a_j e^{ijx}, \]
	where $|a_j|=o(|j|^{-m})$ as $|j|\to\infty$, for any $m>0$. 
	
	Next, put
	\[ h_+(x):= \sum_{j\geq 0} a_j e^{-jx},\quad h_-(x) =\sum_{j> 0} a_{-j} e^{-jx}.\]
	Since $|a_j|=o(|j|^{-m})$ as $|j|\to\infty$, for any $m>0$, the functions $h_\pm$ belongs to  $C^\infty\left(\overline{\mathbb{C}_+}\right)\cap \text{Hol}\left(\mathbb{C}_+\right)$, where $\mathbb{C}_+:=\{z\in\mathbb{C}:\re(z)>0\}$. So, as in the previous example, by  Cauchy's  estimates for the derivatives of analytic functions, we have $h_\pm\in C^{M}_1\left([0,1]\right)$ with $M_n=n!$. Moreover, since {$h(x)=h_+(-ix)+h_-(ix)$} for any $x\in\mathbb{R}$,  either $h_+$ or $h_-$ satisfies \eqref{eq: func with large der}.
\end{Ex}

\section{Theorem \ref{thm: thm1} and Proposition \ref{prop: extra smooth}}
Theorem  \ref{thm: thm1} immediately follows from the next lemma.
\begin{lemma}\label{lem: power-sub estimate}
	Let $M$ be a positive sequence, $k>1$ be an integer, and  $f\in C^\infty[0,1]$ be a function such that  $ \max_{[0,1]} |f^{(n)}(x)|\leq M_n$ for any $n\geq 0$. Put $g(x)=f(x^{1/k})$. If $g\in C^\infty[0,1]$, then
	\[|g^{(n)}(x)|\leq \frac{2^n M_n}{x^{(1-1/k)n}},\quad n\geq 0.  \]
\end{lemma}
\begin{proof}
	Fix $n\geq 1$, First we write Taylor expansion to  $f$ around the origin with integral remainder:
	$$f(x)=\sum_{j=0}^{n-1} \frac{f^{(j)}(0)}{j!}x^j+\frac{1}{(n-1)!}\int_0^x f^{(n)}(t)(x-t)^{n-1}dt.  $$
	Since $g$ is   $C^\infty[0,1]$ function, we have $f^{(j)}(0)=0$ for $j$ which is not divisible by $k$. Therefore
	\begin{equation}
	g(x)=P(x)+\frac{1}{(n-1)!}\int_0^{x^{1/k}} f^{(n)}(t)(x^{1/k}-t)^{n-1}dt, \label{eq: g expan}
	\end{equation}
	where $P$ is a polynomial of degree at most $\tfrac{n-1}{k}$. Put 
	$$F(x,t) =\frac{1}{(n-1)!} (x^{1/k}-t)^{n-1}=\frac{1}{(n-1)!}\sum_{j=0}^{n-1}(-1)^{n-1-j}\binom{n-1}{j}x^{j/k}t^{n-1-j}. $$ 
	Differentiating $F$ $n$ times with respect to the variable  $x$  yields
	$$ \frac{\partial^n}{\partial x^n}F(x,t) =\frac{1}{(n-1)!}\sum_{j=0}^{n-1}\left[\left(\prod_{\ell=0}^{n-1}\left(\frac{j}{k}-\ell\right)\right) (-1)^{n-1-j}\binom{n-1}{j}x^{j/k-n}t^{n-1-j}\right].  $$
	In particular, for $0<t<x^{1/k}$, 
	\begin{equation}
	\left|\frac{\partial^n}{\partial x^n}F(x,t)\right|\leq \sum_{j=0}^{n-1} \binom{n-1}{j} x^{\frac{n-1}{k}-n}= 2^{n-1} x^{\frac{n-1}{k}-n}. \label{eq: F der first}
	\end{equation}
	In addition, the chain rule yields 
	\begin{equation*}
	\frac{\partial^{\ell}}{\partial x^{\ell}}F(x,x^{1/k})=\begin{cases} 
	\left(\frac{x^{\tfrac{1}{k}-1}}{k}\right)^{n-1}, & \ell=n-1\\
	0,\quad & \ell<n-1.
	\end{cases}
	\end{equation*}
	Thus, by differentiating \eqref{eq: g expan} $n$ times, we get
	\begin{align*}
	 g^{(n)}(x)&=\int_0^{x^{1/k}}  \frac{\partial^n}{\partial x^n}F(x,t) f^{(n)}(t)dt+\frac{\partial^{n-1}}{\partial x^{n-1}}F(x,x^{1/k})f^{(n)}(x^{\frac{1}{k}})\frac{x^{\tfrac{1}{k}-1}}{k}\\
	 &=\int_0^{x^{1/k}}  \frac{\partial^n}{\partial x^n}F(x,t) f^{(n)}(t)dt+\left(\frac{x^{1/k-1}}{k}\right)^n f^{(n)}(x^{\frac{1}{k}}).
	\end{align*}
	Finally,  using \eqref{eq: F der first}  and the bound {$|f^{(n)}|\le M_n$} we obtain
	$$ |g^{(n)}(x)|\leq M_n\left(2^{n-1}x^{n/k-n}+\frac{1}{k^n}x^{n/k-n}\right)\leq\frac{2^n M_n}{x^{(1-1/k)n}}. $$
\end{proof}
The next lemma yields Proposition \ref{prop: extra smooth} (by taking $n=k\ell+1$) and it is also being used in the proof of Theorem \ref{thm:thm2}.
\begin{lemma}\label{lem: add smooth lem}
	Let $M$ be a positive sequence such that $(M_n/n!)_n$ is non decreasing, $0<\sigma<1$, and  $g\in C^\infty[0,1]$ be a function such that  $  |g^{(n)}(x)|\leq \frac{M_n}{x^{(1-\sigma)n}}$ for any $n\geq 0$ and $x\in(0,1]$. Then for any $0\leq\ell\leq n$ and $0\leq x\leq 1$ we have
	\[ \left|g^{(\ell)}(x)\right|\leq 2^n \frac{\ell!}{n!} M_n \cdot \begin{cases}
	{n\cdot}\frac{(\ell-\sigma n)^{-1}}{x^{\ell-\sigma n}}, &\ell>\sigma n\\
	1+n\cdot \log\frac{1}{x},& \ell=\sigma n\\
	{n\cdot} (\sigma n-\ell)^{-1},&  \ell<\sigma n.
	\end{cases} \]
\end{lemma}
\begin{proof}
	The case $\ell=n$ is trivial. Assume that $\ell<n$.
	Writing Taylor expansion of degree $n-\ell-1$ to the function $g^{(\ell)}$ around 1 with integral remainder, we get
	$$  g^{(\ell)}(x)=\sum_{j=0}^{n-\ell-1}\frac{g^{(\ell+j)}(1)}{j!}(x-1)^j+\frac{1}{(n-\ell-1)!}\int_1^x g^{(n)}(t)(t-x)^{n-\ell-1} dt. $$
	For $x\in(0,1]$, 
	\begin{align*}
	\left|\int_1^x g^{(n)}(t)(t-x)^{n-\ell-1}  dt\right|\leq \int_x^1\left|g^{(n)}(t)\right| (t-x)^{n-\ell-1} dt&\leq M_{n}\int_x^1 (t-x)^{n-\ell-1} t^{-(1-\sigma)n}dt\\&\leq M_{n}\int_x^1  t^{\sigma n-\ell-1}dt.
	\end{align*}
	Hence
	\begin{align*}
	\left|g^{(\ell)}(x)\right|&\leq \sum_{j=0}^{n-\ell-1}\frac{M_{\ell+j}}{j!}+\frac{M_{n}}{(n-\ell-1)!}\int_x^1  t^{\sigma n-\ell-1}dt \\ &\leq \sum_{j=0}^{n-\ell-1}\frac{(\ell+j)!}{j!}\frac{M_{n}}{n!}+n 2^n \ell!\frac{M_{n}}{n!}\int_x^1  t^{\sigma n-\ell-1}dt \\
	&\leq\ell!\frac{M_n}{n!}\left(\sum_{j=0}^{n-\ell-1} \binom{n}{j}+n2^n\int_x^1  t^{\sigma n-\ell-1}dt\right) \\
	&\leq 2^n\ell!\frac{M_n}{n!} \left(1+n\int_x^1 t^{\sigma n-\ell-1}dt\right).
	\end{align*}
	Since
	\[\int_x^1 t^{\sigma n-\ell-1}dt \leq 
	\begin{cases}
(\ell-\sigma n)^{-1}x^{\sigma n-\ell}, &\ell>\sigma n\\
	\log\frac{1}{x},&\ell=\sigma n\\
	(\sigma n-\ell)^{-1},&\ell<\sigma n,
	\end{cases}
	\]	
	we obtained the desired bound.
\end{proof}

\section{Theorem~\ref{thm:thm2}}
\begin{lemma}\label{lem: powersub lem}
	Let $M$ be a positive sequence such that $(M_n/n!)_n$ is non decreasing, $1<k$ be an integer, and  $g\in C^\infty(0,1]$ be a function such that  $  |g^{(n)}(x)|\leq \frac{M_n}{x^{(1-1/k)n}}$ for any $n\geq 0$ and $x\in (0,1]$. Put $f(x)=g(x^k)$. Then there exist constants $A,\; C>0$ such that 
	\[ \left|f^{(n)}(x)\right|\leq A C^n M_n  \begin{cases}
	1, &k\nmid n \text{ or } n=0;\\
	1+\log\frac{1}{x},& k\mid n.
	\end{cases} \]
\end{lemma}
\begin{proof}
	First we argue by induction on $n$ that 
	\begin{equation}
	f^{(n)}(x)= \sum_{ \substack{i+j=n\\ 1\leq i\leq n,\; i(k-1)\geq j  }}B_{n}(i,j)g^{(i)}(x^k)x^{i(k-1)-j} \label{eq: g der via f}
	\end{equation}
	where \[ |B_{n}(i,j)|\leq  C^n n^{n-i}.\]
	Indeed, 
	$$ \frac{d}{dx}\left(g^{(i)}(x^k)x^{i(k-1)-j}\right)=kg^{(i+1)}(x^k)x^{(i+1)(k-1)-j}+ (i(k-1)-j)x^{i(k-1)-(j+1)}g^{(i)}(x^k)$$
	implies that 
	$$B_{n+1}(i,j)=kB_n(i-1,j)+\left(i(k-1)-j\right)B_n(i,j-1).$$
	So making use of the induction hypothesis, we find that 
	$$\left|B_{n+1}(i,j)\right|\leq k C^n  n^{n+1-i}+(ik-n) C^n n^{n-i}\leq C^{n+1}(n+1)^{n+1-i} $$
	as claimed. 
	
	Using \eqref{eq: g der via f}, we find that 
	$$\left|f^{(n)}(x)\right|\leq C^n \sum_{n/k\leq i\leq n}n^{n-i}\left|g^{(i)}(x^k)x^{ik-n}\right| $$
	By Lemma \ref{lem: add smooth lem}, 
	$$ \left|g^{(i)}(x^k)x^{ik-n}\right| \leq 2^n \frac{i!}{n!}M_n \begin{cases}
	1+n \log\frac{1}{{x^k}},& i=\frac{n}{k}\\
	{n},& i>\frac{n}{k}.
	\end{cases} $$
	Thus in the case $k\mid n$, we find that 
\begin{align*}
\left|f^{(n)}(x)\right|&\leq {n\cdot}(2C)^n M_n{\bigg[} \sum_{n/k< i\leq n}\frac{n^{n-i} i!}{n!}+\frac{n^{n(1-1/k)} \left(\tfrac{n}{k}\right)!}{n!}\left(1+\log\frac{1}{{x^k}}\right){\bigg]}\\ &\leq {nk\cdot}(2eC)^n M_n \left(1+n\log\frac{1}{x}\right), 
\end{align*}
	while in the case $k\nmid n$, we have 
	$$  \left|f^{(n)}(x)\right|\leq {n}(2C)^n M_n \sum_{n/k< i\leq n}\frac{n^{n-i} i!}{n!}\leq {nk}(2eC)^n M_n.  $$
\end{proof}

\begin{proof}[Proof of Theorem~\ref{thm:thm2}]  
	Without loss of generality we assume that  $I\subseteq [0,\infty)$ .
Assume first that  $\sum_{n\geq 0} \frac{M^C_n}{M^C_{n+1}}=\infty$.       Let $g\in C^M_a(I)$ with $g^{(n)}(x_0)\equiv 0$. We need to show that $g\equiv0$. If $x_0\neq 0$, then  $g\in C^M\left(I\setminus\{0\}\right)$, in particular by the Denjoy--Carleman Theorem $g\equiv 0$. On the other hand, if $x_0=0$, then by replacing $g(x)$ with $g(x/C)$ with sufficiently large $C>0$, there is no loss of generality with assuming that $[0,1]\subseteq I$. Let $k\in\mathbb{N}$ such that $1-\tfrac{1}{k}>a$.

If $M$ is log-convex,  denote by $\widehat{M}$ the log-convex sequence defined by
$$\widehat{M}_0=M_0,\quad \frac{\widehat{M}_{n-1}}{\widehat{M}_n}=\min \left\{\frac{M_{n-1}}{M_n}\;,\; \frac{1}{n} \right\}. $$
Note that $\widehat{M_n}\geq M_n$ for any $n\geq0$. Moreover,
$$ \frac{\widehat{M}_n}{n!}=M_0 \prod_{j=1}^n \max\left\{\frac{ M_j}{{j}M_{j-1}}\;,\; 1\right\} $$, so the sequence $({\widehat{M}_n}/{n!})_{n\geq 0}$ is non-decreasing.  By the condensation test for convergence, 
$$\sum_{n\geq 0} 2^n\frac{M_{2^n-1}}{M_{2^n}}=\infty\quad \Rightarrow \quad \sum_{n\geq 0} \min\left\{2^n\frac{M_{2^n-1}}{M_{2^n}}\;,\;1\right\}=\infty \quad \Rightarrow   \sum_{n\geq 0}\frac{\widehat{M}_{n}}{\widehat{M}_{n+1}}=\infty.$$
On the other hand, if the sequence $(M_n/{n!})_{n\geq 0}$ is non decreasing, then we put $\widehat{M}=M$.

In both cases, the function   $g$ belongs to $ C^{\widehat{M}}_{1-1/k}\left([0,1]\right)$, and $\widehat{M}$ satisfies the assumptions of Lemma~\ref{lem: powersub lem} and $\sum_{n\geq 0} \frac{\widehat{M}^C_n}{\widehat{M}^C_{n+1}}=\infty$. Consider the function $h(x)=\int_0^x g(y^k)dy$. By  Lemma \ref{lem: powersub lem}, $h\in C^{\widehat{M}}([0,1])$ and since $g^{(n)}(0)\equiv 0$, then also $h^{(n)}(0)\equiv 0$. By the Denjoy--Carleman Theorem, $h\equiv 0$ and therefore also $g\equiv 0$, as claimed.  

Next, if $\sum_{n\geq 0} \frac{M^C_n}{M^C_{n+1}}<\infty$, then by Denjoy--Carleman Theorem, the class $C^M(I)$ is not quasianalytic. Thus there exists a  non-zero function $g\in C^{M}(I)$ which is compactly supported in $I\setminus\{0\}$.  Such a function $g$ belongs to the class    $C_a^M(I)$, therefore the latter is not quasianalytic.
\end{proof}

\section{The multivariate  case}
\subsection{Carleman classes.}
Here we will use standard multiindex notation: If $\alpha=(\alpha_1,\cdots,\alpha_d)\in \mathbb{Z}_+^d$ and $x=(x_1,\cdots,x_d)\in\mathbb{R}^d$, we write $|\alpha|:=\alpha_1+\cdots+\alpha_d $,  and $\displaystyle \frac{\partial^{|\alpha|}}{\partial x^{\alpha}}:=\frac{\partial^{\alpha_1+\cdots+\alpha_d}}{\partial x_1^{\alpha_1}\cdots \partial x_d^{\alpha_d}}$.
\begin{definition}\label{def: multCarlemanClass}
	Let $M$ be a sequence of positive numbers   which is logarithmically convex, and let $\Omega$ be a connected set in $\mathbb{R}^d$. \textit{The  Carleman class} $C^M(\Omega)$ consists of all functions $f\in C^\infty(\Omega)$ such that  for any compact $K\subset\Omega$, there exist $A,\;B>0$ such that 
	\[\left|\frac{\partial^{|\alpha|}}{\partial x^\alpha}f(x)\right|\leq A B^{|\alpha|} M_{|\alpha|} ,\quad x\in \Omega \]
	and for any multiindex $\alpha=(\alpha_1,\cdots,\alpha_d)\in \mathbb{Z}_+^d$.
\end{definition}
\begin{definition}
	Let $M$ be a  sequence of positive numbers,  $\Omega$ be a connected set in $\mathbb{R}^d$, and let $a=(a_1,\cdots,a_d)\in \mathbb{R}_+^d$. The class  $C_a^M(\Omega)$ consists of all functions $g\in C^\infty(\Omega)$ such that  for any compact $K\subset\Omega$, there exist $A,\;B>0$ such that 
	\[\left|\frac{\partial^{|\alpha|}}{\partial x^\alpha}g(x)\right|\leq A B^{|\alpha|}\frac{M_{|\alpha|}}{|x_1|^{a_1\cdot\alpha_1}|x_2|^{a_2\cdot\alpha_2}\cdots|x_d|^{a_d\cdot\alpha_d}}\]
	 for any {$x\in K\backslash\{x:x_1\cdots x_d=0\}$} and  any multiindex $\alpha=(\alpha_1,\cdots,\alpha_d)\in \mathbb{Z}_+^d$. 
\end{definition}
\begin{definition}
A set $\Omega\subset\mathbb{R}^d$ is called star-shaped (with respect to the origin)	if for any $x\in \Omega$,  $\{tx:t\in[0,1]\}\subset\Omega$. 
\end{definition}
\subsection{Results.}
The next result is a muliivariate version of Theorem \ref{thm: thm1}.
\begin{theorem} \label{thm: thm3}
	Let $C^M([0,1]^d)$ be a  quasianalytic Carleman class. For  $k\in\mathbb{N}^d$,    
	denote by $y_k:\mathbb{R}^d\to\mathbb{R}^d$ the map defined by $y_k(x)=(x_1^{k_1},\cdots,x_d^{k_d})$. If $g\in C^\infty([0,1]^d)$ is such that $f = g \circ y_k\in C^M([0,1]^d)$, then $g\in C^M_a([0,1]^d)$ with $a=(\tfrac{k_1-1}{k_1},\cdots,\tfrac{k_d-1}{k_d})$.
\end{theorem}
Theorem \ref{thm: thm3} is a special case of the following claim: for any  $1\leq \ell\leq d$ and  multiindex $\alpha=\left(\alpha_1,\cdots,\alpha_d\right)$, we have 
\[\left|\frac{\partial^{|\alpha|}}{\partial x^\alpha}g(x_1,\cdots,x_\ell,x_{\ell+1}^{k_{\ell+1}},\cdots,x_d^{k_d})\right|\leq A B^{|\alpha|}\frac{M_{\alpha_1}\cdots M_{\alpha_d}}{|x_1|^{a_1\cdot\alpha_1}|x_2|^{a_2\cdot\alpha_2}\cdots|x_\ell|^{a_\ell\cdot\alpha_\ell}},\]
where $a=(\tfrac{k_1-1}{k_1},\cdots,\tfrac{k_d-1}{k_d})$.
This claim is an immediate consequence of Theorem \ref{thm: thm1}, by induction on $\ell$ and $d$.

The next result is the multivariate version of Theorem \ref{thm:thm2}, and it follows form it  by restricting functions from $C^M_a(\Omega)$ to lines. 
\begin{theorem}\label{thm:thm4}
	Let $M$ be a positive sequence,  $a\in [0,1)^d$, and and  $\Omega\subset \mathbb{R}^d$ be  star-shaped such that $\Omega\cap \{x_1\cdots
	x_d\neq 0\}$ is dense in $\Omega$. If $M$ is log--convex or $(M_n/n!)_{n\geq0}$ is non-decreasing, then the class $C^M_a(\Omega)$ is quasianalytic if and only if 
\eqref{eq:dc} holds.

\end{theorem}

\paragraph*{Acknowledgement} This work was started during the visit of the authors to IMPAN, Warsaw, in the framework of workshop ``Contemporary quasianalyticity problems''. The authors thank Misha Sodin for encouragement and useful remarks regrading the presentation of this paper. We are grateful to Gerhard Schindl for spotting several lapses in the published paper, which are corrected in the current arXiv version.

\bibliographystyle{plain}
\bibliography{bibPSiCC}
\today

\end{document}